\documentclass[11pt]{amsart}

\usepackage[bookmarks=true]{hyperref}

\usepackage{mathptmx}
\usepackage[english]{babel}
\usepackage{amsmath}
\usepackage{amssymb}
\usepackage{amsthm}
\usepackage{array}
\usepackage[all,tips]{xy}
\usepackage{enumerate}
\usepackage{graphicx}
\usepackage{lmodern}
\usepackage{mathrsfs}

\newtheorem{theorem}{Theorem}

\newtheorem{prop}[theorem]{Proposition}
\newtheorem{cor}[theorem]{Corollary}

\theoremstyle{definition}

\DeclareMathOperator{\Ort}{O}
\DeclareMathOperator{\Spin}{Spin}
\DeclareMathOperator{\SO}{SO}
\DeclareMathOperator{\PO}{PO}

\DeclareMathOperator{\PSO}{PSO}

\DeclareMathOperator{\Isom}{Isom}

\newcommand{\Z}{\mathbb Z}
\newcommand{\Q}{\mathbb Q}
\newcommand{\R}{\mathbb R}

\newcommand{\Hy}{\mathbf H}
\newcommand{\II}{\mathrm{II}}
\newcommand{\I}{\mathrm{I}}
\newcommand{\Lat}{\mathcal{L}}
\newcommand{\mass}{\mathrm{mass}}
\newcommand{\vol}{\mathrm{vol}_\Hy}

\newcommand{\G}{\mathrm G}
\newcommand{\bG}{\overline{\G}}

\newcommand{\Buil}{\mathcal B}

\newcommand{\bs}{\backslash}

\begin{document}
\title{Even unimodular Lorentzian lattices and hyperbolic volume}

\author{Vincent Emery}

\begin{abstract}
We compute the hyperbolic covolume of the automorphism
group of each even unimodular Lorentzian lattice. The result is obtained as a
consequence of a previous work with Belolipetsky, which uses Prasad's
volume to compute the volumes of the smallest hyperbolic arithmetic
orbifolds.
\end{abstract}

\address{
Max Planck Institute for Mathematics\\
Vivatsgasse 7\\
53111 Bonn\\
Germany
}
\email{vincent.emery@gmail.com}

\date{\today}


\maketitle

\section{Introduction}

Let $\Hy^n$ be the hyperbolic $n$-space, of constant curvature $-1$. We
denote by $\Isom(\Hy^n)$ the group of isometries of $\Hy^n$. 
One way to construct a lattice in $\Isom(\Hy^n)$ is to
consider the automorphism group $\Ort(L)$ of a Lorentzian lattice
$L \subset \R^{n,1}$. Of particular
interest are the unimodular Lorentzian lattices. There exist two such
types of lattices: the odd unimodular Lorentzian lattice $\I_{n,1}$ and the even unimodular
Lorentzian lattice $\II_{n,1}$. Their study
appears in connection with the study of Euclidean lattices, as shown in
the book of Conway and Sloane \cite{Conway-Sloane}. While $\I_{n,1}$ exists for every
dimension $n$, the even lattice $\II_{n,1}$ exists only when $n \equiv 1
\mod 8$.

In
\cite{BelEme} (see also \cite{EmePhD}) the following theorem was proved.

\setcounter{theorem}{-1}

\begin{theorem}
	\label{thm:smallest-volumes}
  For each odd dimension $n = 2r-1 \ge 5$, there is  a unique orientable
  non-compact arithmetic hyperbolic $n$-orbifold $\Delta_n \bs \Hy^n$ of the smallest
  volume (with $\Delta_n$ an arithmetic lattice of
  $\Isom(\Hy^n)$). Its volume is given by:
  \begin{align}
      \frac{1}{2^{r-2}} \; \zeta(r) \; \prod_{j=1}^{r-1} \frac{(2j -1)!}{(2
      \pi)^{2j}} \zeta(2j) &\qquad \mbox{if }  n \equiv 1 \mod 8 ;\label{vol-min-II}\\
       \frac{(2^r  -1)  (2^{r-1}
	-1)}{3   \cdot  2^{r-1}}   \;  \zeta(r)   \;  \prod_{j=1}^{r-1}
      \frac{(2j -1)!}{(2 \pi)^{2j}} \zeta(2j) &\qquad \mbox{if } n
      \equiv 5 \mod 8; \label{vol-min-5}
      \\
      \frac{3^{r-1/2}}{2^{r-1}} \; L_{\Q(\sqrt{-3})|\Q}\!(r) \; \prod_{j=1}^{r-1}
      \frac{(2j -1)!}{(2 \pi)^{2j}} \zeta(2j)
      &\qquad \mbox{if } n \equiv 3 \mod 4. \label{vol-min-7}
  \end{align}
\end{theorem}
\medskip

It is remarkable that the smallest volume has the simplest
form~\eqref{vol-min-II} exactly for the dimensions $n$ where the even
unimodular Lorentzian  lattice $\II_{n,1}$ exists.  The main purpose of this article is to show that
for these $n$ the arithmetic group $\Delta_n$ of
Theorem~\ref{thm:smallest-volumes} is actually given by the group
$\SO(\II_{n,1})$ of special automorphisms of $\II_{n,1}$ (cf.~Theorem~\ref{thm:identification}).  
In particular, this allows to deduce in
Corollary~\ref{cor:volume-aut-even-unimod} the hyperbolic covolume of
the automorphism group $\Ort(\II_{n,1})$. This complements the work of
Ratcliffe and Tschantz \cite{RatTsch97}, where the covolume of
$\Ort(\I_{n,1})$ was determined for every $n$.

In \S\ref{sec:volume-Coxeter}--\ref{sec:disussion-cusps} we discuss some
interesting consequences of our main result. Finally, in
\S\ref{sec:odd-unimod-sqrt3} we discuss the case of
Formulas~\eqref{vol-min-5}--\eqref{vol-min-7}. In particular, we state
in Proposition~\ref{prop:odd-unimod-index-3} the exact relation between
$\Delta_n$ and $\Ort(\I_{n,1})$ when $n \equiv 5 \mod 8$.

\subsection*{Acknowledgements}

I would like to thank Curtis McMullen for the interesting discussions
that are at the origin of this article. I thank Jiu-Kang Yu for his help
concerning Bruhat-Tits theory, Steve Tschantz for the numerical
computation mentioned in \S \ref{sec:volume-Coxeter}, Anna Felikson and
Pavel Tumarkin for helpful discussions, and Ruth
Kellerhals, John Ratcliffe and the referee for helpful comments. I am thankful
to the MPIM in Bonn for the hospitality and the financial support.

\section{Main result and its proof}
\label{sec:identification}

For $n \equiv 1 \mod 8$, we consider the even unimodular lattice
$\II_{n,1}$ embedded in the real quadratic space equipped with the
standard rational quadratic form:
\begin{align}
	q(x) &=	-x_0^2 + x_1^2 + \cdots + x_n^2. 
	\label{eq:quad-form-stand}
\end{align}
The group of automorphisms of this quadratic space acts then
isometrically on $\Hy^n$, via an identification of $\Hy^n$ with its
projective model. The group $\Ort(\II_{n,1})$
(resp. $\SO(\II_{n,1})$) of  automorphisms (resp. special automorphisms)
preserving $q$ and the lattice $\II_{n,1}$  acts discontinuously  on
$\Hy^n$. More precisely, the group $\PO(\II_{n,1}) =
\Ort(\II_{n,1})/\{\pm I \}$ (resp. $\PSO(\II_{n,1}) =
\SO(\II_{n,1})/\{\pm I \}$), where $I$ is the identity matrix, can be
seen as a discrete subgroup of $\Isom(\Hy^n)$.

\begin{theorem}
	\label{thm:identification}
	For $n \equiv 1 \mod 8$, the group $\Delta_n$ is conjugate in
	$\Isom(\Hy^n)$ to $\PSO(\II_{n,1})$.
\end{theorem}

\begin{proof}
	We denote by $V$ the quadratic space over $\Q$ equipped with quadratic
	form $\frac{1}{2}q$, where $q$ is given in
	\eqref{eq:quad-form-stand}. 
	Let $\G$ be the algebraic group defined over $\Q$ with
	$\G(\Q) = \Spin(V)$, the group of spinors of $V$.
	Let $\bG$ be the adjoint form of $\G$. Then $\bG(\R)$ is isomorphic
	to $\Isom(\Hy^n)$. For each prime $p$ we consider the quadratic
	space $V_p = V \otimes_\Q \Q_p$, and the Bruhat-Tits building $\Buil_p$
	associated with $\Spin(V_p)$ and $\SO(V_p)$. Note that $\G$ and
	$\bG$ are split over $\Q_p$, for every prime $p$ (cf.~\cite[Prop.~3.9]{BelEme}).

	Let $L$ be the lattice in $V$ that identifies to $\II_{n,1}$ via
	the embedding in the quadratic space $(V \otimes_\Q \R, \;q)$.
	For each prime $p$, we consider the lattice $L_p = L \otimes
	\Z_p$, which is a maximal lattice in $V_p$ \cite[\S 5]{BeloGan05}. 
	Bruhat-Tits theory allows to identify
	the lattice $L_p$ as an hyperspecial point of the building
	$\Buil_p$ (cf. \cite[\S 5]{BeloGan05}), whose stabilizer in
	$\SO(V_p)$ is $\SO(L_p)$.

	Let us denote by $K_p$ the hyperspecial parahoric subgroup of $\Spin(V_p)$ 
	that stabilizes $L_p \in \Buil_p$. The set of all these $K_p$ for $p$ prime is a 
	coherent collection of parahoric subgroups, and this defines a principal
	arithmetic subgroup of $\G(\Q)$ (see \cite[\S 2.2]{BelEme} for details):
	\begin{align}
		\Lambda &= \G(\Q) \cap \prod_p K_p\;,
		\label{eq:princal-arithm}
	\end{align} 
	which by construction maps into $\SO(L) = \SO(\II_{n,1})$.
	But $\Lambda$ corresponds exactly to the group $\Lambda_1$ in
	\cite{BelEme}, whose image in $\bG(\R)$ gives the group
	$\Delta_n$. It was proved in \cite{BelEme} that $\Lambda_1$
	is maximal, and that up to conjugacy its construction does not
	depend on the choice of a coherent collection of hyperspecial
	subgroups. It follows that $\PSO(\II_{n,1})$ is conjugate to
	$\Delta_n$ in $\Isom(\Hy^n)$.
\end{proof}

From Theorem~\ref{thm:smallest-volumes}  and~\ref{thm:identification} we obtain  the covolume of the group
$\PO(\II_{n,1})$, which contains $\PSO(\II_{n,1})$ as a subgroup of index two. 
In order to simplify even more the volume formula, we
use the well-known expression of $\zeta(2j)$ in terms of the Bernoulli number
$B_{2j}$.

\begin{cor}
	The covolume of the action of $\PO(\II_{n,1})$ on $\Hy^n$ equals
  \begin{align}
	  \zeta(r) \prod_{j=1}^{r-1} \frac{\vert B_{2j}\vert}{8j} \;,
	  \label{vol-aut-even-unimod}
  \end{align}
	where $B_k$ is the $k$-th Bernoulli number. 
	\label{cor:volume-aut-even-unimod}
\end{cor}

\section{Volume of Coxeter polytopes}
\label{sec:volume-Coxeter}

Corollary~\ref{cor:volume-aut-even-unimod} was already known in dimension
$n=9$ (see~\S \ref{sec:odd-unimod-sqrt3}), where $\PO(\II_{9,1})$
is the Coxeter group generated by reflections through the faces of a
simplex. The only other group $\PO(\II_{n,1})$ that is reflective is
$\PO(\II_{17,1})$, as it follows form the work of Conway and Vinberg
(cf. \cite[Ch.~27]{Conway-Sloane} and \cite[Part II Ch.~6 \S
2.1]{Vinb93}). It contains as a subgroup of index two the following
Coxeter group:
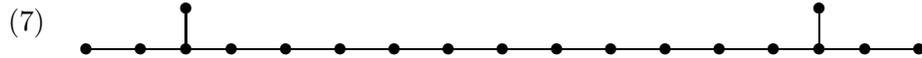
\begin{equation}
	\label{eq:Coxeter-17}
	\begin{split}
	\xymatrix@C=14pt@R=12pt{
	& & *={\bullet} \ar@{-}[d]  & & & & & & & & & & & & *={\bullet} \ar@{-}[d] & &  \\
	*={\bullet} \ar@{-}[r]&
	*={\bullet} \ar@{-}[r]&
	*={\bullet} \ar@{-}[r]&
	*={\bullet} \ar@{-}[r]&
	*={\bullet} \ar@{-}[r]&
	*={\bullet} \ar@{-}[r]&
	*={\bullet} \ar@{-}[r]&
	*={\bullet} \ar@{-}[r]&
	*={\bullet} \ar@{-}[r]&
	*={\bullet} \ar@{-}[r]&
	*={\bullet} \ar@{-}[r]&
	*={\bullet} \ar@{-}[r]&
	*={\bullet} \ar@{-}[r]&
	*={\bullet} \ar@{-}[r]&
	*={\bullet} \ar@{-}[r]&
	*={\bullet} \ar@{-}[r]& *={\bullet}
				}
	\end{split}
\end{equation}

\vspace{10pt}

\begin{cor}
	\label{cor:volume-Coxeter-gp}
	Let $P \subset \Hy^{17}$ be a Coxeter polytope corresponding to
	the diagram~\eqref{eq:Coxeter-17}. Then
	\begin{align*}
		\vol(P) &= \frac{691 \cdot 3617}{2^{38}
		\cdot 3^{10} \cdot 5^4 \cdot 7^2 \cdot 11 \cdot 13 \cdot
		17} \; \zeta(9)
		\\
		&\approx  2.072451981 \cdot 10^{-18}.
	\end{align*}
\end{cor}
It is not clear how one could compute precisely the volume of such an
high- and odd-dimensional hyperbolic polytope without an identification
with the fundamental domain of an arithmetic group. Steve Tschantz was 
able to compute the following numerical approximation, which agrees with
the result of Corollary~\ref{cor:volume-Coxeter-gp}.
\begin{align}
	\vol(P) &= 2.069 \cdot 10^{-18} \; \pm \; 2.4 \cdot 10^{-20}.
	\label{eq:Tschantz}
\end{align}
The computation took about 60 hours, showing that for this kind  of 
polytopes even numerical computation  is not an easy task. 

It is most likely that $P$ realizes the smallest volume among all
hyperbolic Coxeter polytopes (non-compact or not), independently of the
dimension. The results of \cite{EmePhD,BelEme} (odd dimensions) and
\cite{Belo04} (even dimensions) determine the smallest possible
arithmetic orientable hyperbolic orbifolds. From them we see that a
Coxeter polytope smaller than $P$ and being the fundamental domain of an
arithmetic group must necessarily lie in $\Hy^{17}$, be commensurable to
$P$, and have exactly half of the volume of $P$. We don't know if such a
Coxeter polytope could exist. 

The small size of $P$ can also be
explained by Schl\"afli differential formula for the volume of polytopes
(see \cite[Part~I Ch.~7 \S 2.2]{Vinb93}).
According to this formula, the volume of Coxeter polytopes tends to be
smaller for polytopes having large dihedral angles. The small size of $P$ 
results then from the combination of
two factors: the only dihedral angles in $P$ are $\pi/2$ and
$\pi/3$; and relatively to its dimension, $P$ is determined by few
hyperplanes (actually the smallest possible number in $\Hy^{17}$).
These two conditions are a very rare occurrence in high dimensions.

\medskip

\section{Comparison with the mass formula}
\label{sec:disussion-cusps}

The lattice $\II_{25,1}$ plays an important role in connection with the
study of even unimodular Euclidean lattices in dimension $24$
(see~\cite[Theorem~5, Ch.~26]{Conway-Sloane}). For $n \equiv 1 \mod 8$,
let $\Lat_{n-1}$ denotes the set (up to isomorphism) of
$(n-1)$-dimensional even unimodular Euclidean lattices. This is a finite
set, and an important invariant is its \emph{mass}, defined as
\begin{align}
	\mass(\Lat_{n-1}) &=  \sum_{L \in \Lat_{n-1}} \frac{1}{\vert
	\Ort(L) \vert}\;.
	\label{eq:mass-definition}
\end{align}
For $n=9, 17$ and $25$ each group $\Ort(L)$ (with $L \in
\Lat_{n-1})$ appears as a subgroup of $\Ort(\II_{n,1})$ as the
stabilizer of a point at infinity of $\Hy^n$. Therefore, the groups
$\Ort(L)$ correspond to cusps of the hyperbolic orbifold defined by
$\Ort(\II_{n,1})$, and $\mass(\Lat_{n-1})$ could be regarded as a
measurement of the contribution from these cusps to the volume. It is
then quite natural to consider the ratio ``covolume of
$\Ort(\II_{n,1})$ divided by $\mass(\Lat_{n-1})$''. From the mass
formula~\cite[Theorem~2, Ch.16]{Conway-Sloane} we obtain the rather
simple formula:
\begin{align}
	\frac{\mbox{covolume of } \Ort(\II_{n,1})}{\mass(\Lat_{n-1})} &=
	2^{-r}
	\frac{\vert B_{2r-2} \vert}{\vert B_{r-1} \vert}
	\zeta(r).
	\label{eq:ratio-volume-mass}
\end{align}
Note that this ratio goes quickly to $\infty$ when $r$ grows. 

We refer
to \cite{StoverEnds} for more precise results on the behaviour of cusps of arithmetic
orbifolds with respect to the dimension.


\medskip

\section{The case of the other odd dimensions}
\label{sec:odd-unimod-sqrt3}

The covolume of $\PO(\I_{n,1})$ was computed by
Ratcliffe and Tschantz in all dimensions $n > 1$ \cite{RatTsch97}.
They obtain the result by evaluating a formula due to Siegel. Note that Prasad's
volume formula, the main ingredient to obtain
Theorem~\ref{thm:smallest-volumes}, may be considered as a far-reaching
extension of this formula of Siegel. Using the fact that
$\PO(\II_{n,1})$ and $\PO(\I_{n,1})$ are commensurable, Ratcliffe and
Tschantz could also deduce the covolume of $\PO(\II_{9,1})$ (cf.
\cite[p. 345]{Johnson-al-99}). By the work
of Vinberg and Kaplinskaya, the group $\PO(\I_{n,1})$ is known to be reflective for $n \le 19$, and combining
this fact with the work of Ratcliffe and Tschantz one can obtain the volume
of several Coxeter polytopes. 

By its construction in~\cite{BelEme}, it is clear that for $n \equiv 5
\mod 8$ the arithmetic group $\Delta_n$ of
Theorem~\ref{thm:smallest-volumes} is commensurable to
$\PSO(\I_{n,1})$. Moreover, we can see that the
ratio of the covolumes of these two groups is equal to $3$. In fact, using
\cite[Prop.~5.9]{BeloGan05} and the same kind of argument as in the proof
of Theorem~\ref{thm:identification}, we get the following result.
It agrees with known facts about simplices in dimension $5$
(cf.~\cite[\S 5]{Johnson-al-99}). 

\begin{prop}
	\label{prop:odd-unimod-index-3}
	For $n \equiv 5 \mod 8$, the group $\PSO(\I_{n,1})$ is conjugate in $\Isom(\Hy^n)$ to a subgroup of index
	$3$ in $\Delta_n$.
\end{prop}

\medskip
For $n \equiv 3 \mod 4$, the group $\Delta_n$ is not commensurable to
$\PSO(\I_{n,1})$. Instead, it is commensurable to the
group $\PO(f,\Z)$ given by the  integral automorphisms of the following 
quadratic form:
\begin{align}
	f &=	-3 x_0^2 + x_1^2 + \cdots + x_n^2.	
	\label{eq:quadr-form-3}
\end{align}
McLeod showed that the group $\PO(f,\Z)$ is reflective when $n
\le 13$ \cite{McLeod11}. Recently, elaborating on their earlier work 
on Siegel's formula (cf.~\cite[pp. 344--345]{Johnson-al-99}), Ratcliffe and Tschantz  
determined the covolume of $\PO(f,\Z)$ (thus obtaining the covolumes of
McLeod's polytopes) \cite{RatTsch12}. For $n \equiv 3 \mod 4$, the ratio between the covolumes of
$\PO(f,\Z)$ and $\Delta_n$ is then computed to be equal to $a(n)/4$,
where $a(n)$ is some odd integer tending to $\infty$ when $n \to \infty$
(see \cite[(35)]{RatTsch12}). An alternative way to obtain the covolume
of McLeod's polytopes would be to determine the relation between
$\PO(f,\Z)$ and $\Delta_n$ in terms of subgroup inclusions, using the
same kind of arguments as for Theorem~\ref{thm:identification} and
Proposition~\ref{prop:odd-unimod-index-3}.

\bibliographystyle{amsplain}
\bibliography{unimod-lattices}

\end{document}